\documentclass[12pt]{amsart}
\advance\hoffset by -0.5 truecm
\usepackage{amssymb}
\newtheorem{Theorem}{Theorem}[section]
\newtheorem{Lemma}[Theorem]{Lemma}
\newtheorem{Corollary}[Theorem]{Corollary}

\newtheorem{Proposition}[Theorem]{Proposition}

\newtheorem{Definition}[Theorem]{Definition}

\newtheorem{Remark}[Theorem]{Remark}

\def\V{\mbox{Var}}

\def\R\re
\def\V{\bf V}

\def \re{{\mathbb R}}

\def \0{\lambda_{0}}

\begin{document}
\title[Solutions of the Yamabe equation]{Multiplicity results for the  Yamabe equation by Lusternik-Schnirelmann theory}

\author[J. Petean]{Jimmy Petean}\thanks{The author is supported by grant 220074 of Fondo Sectorial de Investigaci\'{o}n para la Educaci\'{o}n SEP-CONACYT}
\address{CIMAT  \\
         A.P. 402, 36000 \\
          Guanajuato. Gto. \\
          M\'exico.}
\email{jimmy@cimat.mx}




\begin{abstract} Let $(M,g)$ be any closed Riemannianan manifold  and $(N,h)$ be a Riemannian manifold of
constant positive scalar curvature. We prove that the Yamabe equation on the Riemannian  product $(M\times N , g + \delta h)$ has at
least ${\it Cat} (M) +1 $ solutions for $\delta$ small enough, where ${\it Cat} (M)$ denotes the Lusternik-Schnirelmann-category of $M$. 
$Cat(M)$ of the solutions obtained  have energy arbitrarily close to the minimum.

\end{abstract}

\maketitle

\section{Introduction}Let $(W^{k},g)$ be a closed Riemannian 
manifold of dimension $k \geq 3$. A metric $\overline{g} = f^{p_k -2} g$ conformal to 
$g$ has constant scalar curvature $\mu \in \re$ if and only if the positive  function  $f$ satisfies 
the {\it Yamabe equation} corresponding to $g$:

\begin{equation}
-a_k \Delta_g f + {\bf s}_g f = \mu f^{p_k -1}.
\end{equation}

\noindent
Here $a_k =\frac{4(k-1)}{k-2}$, $p_k = \frac{2k}{k-2}$ is the Sobolev critical exponent and
${\bf s}_g$ denotes the scalar curvature of $g$.

Solutions of the Yamabe equations are critical points of the Hilbert-Einstein functional ${\bf S}$ restricted
to $[g]$,  the conformal class of $g$. If we write $\overline{g} \in [g]$ as  $\overline{g} = f^{p_k -2} g$, we obtain the expression

$${\bf S} (\overline{g} ) = Y_g (f) = \frac{\int_W a_k  \|  \nabla f \|^2 + {\bf s}_g f^2 \  dv_g}{\left( \int_W f^{p_k} \  dv_g \right)^{\frac{2}{p_k}}} ,$$

\noindent
where $dv_g$ denotes the volume element associated to $g$. We call $Y_g$ the $g$-Yamabe functional.

It is a fundamental, well known result,  that for any metric $g$  the infimum of the 
$g$-Yamabe functional  is  achieved (by a smooth positive function)  and therefore there is always at least one (positive)  solution of the Yamabe
equation. We call the infimum the {\it Yamabe constant} of the conformal class $[g]$ and denote it by $Y(M,[g])$.
 When $Y(M,[g]) \leq 0 $  the Yamabe equation has a unique solution,  up to homothecies. 
But in the positive case  there are in general multiple
solutions. The first example of multiplicity to mention, which actually plays  a fundamental part of the theory, is the constant curvature
metric on the sphere: it has a noncompact family of conformal transformations which induce a noncompact family of
solutions of the Yamabe equation (although all the corresponding metrics are of course isometric). Simon  Brendle \cite{Brendle} constructed other metrics
on spheres of high dimensions for which the space of solutions of the Yamabe equation is noncompact. Another interesting general multiplicity result
proved by  Daniel Pollack in \cite{Pollack} is 
that every conformal class with positive Yamabe constant can be $C^0$- approximated
by a conformal class with an arbitrarily large number of (non-isometric) metrics of
constant scalar curvature. 

In this article we will consider Riemannian products of closed manifolds $(M^n ,g)$, $(N^m ,h)$.  Let  $(W^k , g_{\delta} )= (M\times N , g + \delta h)$, where
$k=m+n$ and $\delta \in \re_{>0}.$ This particular case is interesting since it plays an important
role in the study of the Yamabe invariant of a closed manifold, which is defined as the supremum of the Yamabe constants 
(over the family of conformal classes of metrics on the manifold). Many multiplicity results have been obtained for solutions of the Yamabe 
equation corresponding to $g_{\delta}$ in certain particular cases, all of them giving solutions which depend on only 
one of the variables. In all these previous results it is assumed that the scalar curvature of $g$ is also constant and the
results are obtained by applying bifurcation theory. Note that when studying solutions which depend on only one of the variables 
que equation is subcritical, since $p_{m+n} < p_n$.

In the case $(M,g)= (S^n ,  g_0^n )$, where 
$g_0^n $ is the metric of constant curvature 1 on the sphere,
radial solutions of the resulting subcritical  equation have been obtained by Qinian Jin, YanYan Li and Haoyuan  Xu in \cite{Yan}. 
The authors prove that there is a sequence of  positive numbers
$\delta_i \rightarrow 0$ such that for $\delta < \delta_i$, the Yamabe equation corresponding to $g_{\delta}$ 
has at least $i$ different solutions, which are radial functions on $S^n$. 
The authors obtain this result by showing that the $\delta_i$'s are bifurcation instants (local bifurcation) and
then using the global bifurcation theory of Rabinowitz to prove  that the branches of solutions appearing at these bifurcation instants persist to give solutions for every 
$\delta < \delta_i$. Partial results in the same case were obtained by the author in \cite{Petean} using elementary methods. 
The same global bifurcation     techniques were used in  \cite{Henry} to obtain  solutions which are not radial but whose level sets
are any given family of isoparametric hypersurfaces in the sphere.  The local bifurcation theory for the
general case was treated by Levi Lopes de Lima, Paolo Piccione and Michela Zedda in \cite{LPZ2, Piccione}. The authors  prove (under some mild nondegeneracy assumption) that there is a sequence $\delta_i 
\rightarrow 0$ which are degenerate instants, i.e. there exists at least one other solution of the equation
for values of $\delta$ close enough to any of the $\delta_i$'s. Other related results were obtained in \cite{BP1, BP2, Petean2}.

In this article we will obtain the first multiplicity results when the scalar curvature of $g$ is not necessarily constant. Moreover, the results will
be  global in the sense that we obtain solutions for all $\delta >0$ small enough:

\begin{Theorem} Let $(M,g)$ be any closed Riemannian manifold and $(N,h)$ be a Riemannian manifold
of constant positive scalar curvature. There exists $\delta_0 >0$ such that for any $\delta < \delta_0$ there
are at least $Cat(M) +1 $ different solutions of the Yamabe equation for $g+ \delta h$ on $M\times N$.

\end{Theorem}

Recall that the Lusternik-Schnirelmann category of a manifold $M$, $Cat(M)$, is the minimum number
of contractible open subsets needed to cover $M$ (of course it is a topological invariant). $Cat(S^2 )=2$ and 
for any  other compact Riemann surface $S$, $Cat(S) =3$. In particular if $(H,g_h )$ is  a closed hyperbolic surface 
we obtain 4 different solutions of the Yamabe equation on $(H \times {\bf S}^2  , g_h + \delta g^2_0 )$
for $\delta >0$ small (one of them is the constant solution).
These solutions are constant on any slice $\{ x \} \times S^2$. The 3 nonconstant solutions built in the 
theorem  concentrate around 
a point and are certainly cantidates to be minimizers for the Yamabe constant. But it seems interesting
to understand if for $\delta$ small enough one could get solutions concentrating around any point in $H$. 

\vspace{1cm}

The solutions in Theorem 1.1 are functions $u : M \rightarrow \re_{>0}$ considered as functions on $M \times N$. Such a
functions solves the $(g+ \delta h)$-Yamabe equation if and only if it solves 

\begin{equation}
-a_k \Delta_g u + ({\bf s}_g + (1/\delta ) {\bf s} _h ) u = u^{p_k -1}  .
\end{equation}

Therefore Theorem 1.1 is a corollary of:

\begin{Theorem} For any positive constant ${\bf s}_h$ there exists $\delta_0 >0$ such that for any $\delta \in (0, \delta_0 )$ equation
(2) has at least $Cat(M) + 1$ different solutions.

\end{Theorem}

Theorem 1.2 will be proved using Lusternik-Schnirelmann theory. Solutions are obtained as critical points of
a functional $J$ defined on an appropriate Nehari manifold, following ideas introduced by Vieri Benci, Claudio Bonanno and
Anna Maria Micheletti in \cite{Micheletti} which consider the equation with constant coefficients on a 
Riemannian manifold and by 
 Vieri Benci and  G Cerami in
\cite{Benci}, for the same equation in $\re^n$. 
In Section 2  we discuss the preliminary results that give Theorem 1.2 as a corollary of bounding from
below the Lusternik-Schnirelmann category of the space functions in the Nehari manifold where the functional is close to the minimum
by $Cat(M)$.  In section
3  we prove that functions close to the minimum of the functional concentrate around a point.
In section 5 we define the center of mass for functions in a Riemannian manifolds which concentrate on a small
ball, extending the definition of Riemannian center of mass introduced by Karsten Grove and Hermann Karcher in \cite{Grove} for
functions supported in a small ball. This willl be applied in Section 6 to functions in the Nehari manifold to obtain the necessary bound on the
Lusternik-Schnirelmann category. 


\section{Preliminaries}

\subsection{The limiting equation and solution  on $\re^n$}

Let $2<q<p_n$ (if $n=2$ any $q \in (2, \infty )$). It is well known that  there exists a unique (up to translation) positive finite-energy solution $U$  of the equation on $\re^n$

$$-\Delta U + U = U^{q-1}. $$

Moreover the function $U$ is radial (around some chosen point) and it is exponentially decreasing at infinity. Consider  the functional $E: H^1 (\re^n ) \rightarrow \re$,

$$E(f)= \int_{\re^n} (1/2) |\nabla f |^2 + (1/2) f^2 -(1/q) (f^+)^q  \ dx ,$$ 

\noindent 
where $f^+ (x):=\max \{ f(x), 0\}$, and the corresponding Nehari manifold

$$N(E) :=\{ u\in H^1 (M)  - \{ 0 \} : \int_{\re^n} (  | \nabla u |^2 + u^2 ) \ dx = \int_{\re^n}  (u^+ )^q \ dx \} .$$ 

$U$ is actually the minimizer of the functional $E$ restricted to $N(E)$. The minimum is then

\begin{equation}
{\bf m}(E) = \min \{ E(u) \ : \ u \in N(E) \} =  \frac{q-2}{2q} {\| U \|}_q^q .
\end{equation}

For any $\varepsilon >0$ let 

$$E_{\varepsilon} (f)=\varepsilon^{-n}  \int_{\re^n} (\varepsilon^2 /2) |\nabla f |^2 + (1/2) f^2 -(1/q) (f^+)^q  \ dx $$ 

\noindent
and

$$N(E_{\varepsilon} ) :=\{ u\in H^1 (M)  - \{ 0 \} : \int_{\re^n} (  \varepsilon^2  | \nabla u |^2 + u^2 )  \  dx = \int_{\re^n}  (u^+ )^q  \  dx  \} .$$ 

Let $U_{\varepsilon} (x) = U((1/\varepsilon ) x)$.
Then $U_{\varepsilon} \in N(E_{\varepsilon} )$, and it is a solution of

$$-\varepsilon^2 \Delta U_{\varepsilon} + U_{\varepsilon} = U_{\varepsilon}^{q-1}.$$

\noindent
$U_{\varepsilon}$ is a minimizer of $E_{\varepsilon}$ restricted to $N(E_{\varepsilon} )$. Note that the minimum is 
equal to ${\bf m}(E)$.

\subsection{Setting on a Riemannian manifold and the structure of the proof of Theorem 1.1}

Consider  a closed Riemannian manifold $(M^n ,g)$ (of dimension $n$) with scalar curvature ${\bf s}_g$. Let $(N^m , h)$ be a closed Riemannian
manifold  of constant positive scalar curvature ${\bf s}_h$. Then a function $u : M \rightarrow \re$ satisfies the
Yamabe equation corresponding to $(M\times N , g+ \varepsilon^2 h)$ if it solves:

\begin{equation}
-a_{m+n} \Delta_g u + ({\bf s}_g + \varepsilon^{-2} {\bf s}_h ) u =  u^{p_{m+n}-1},  
\end{equation}

\noindent
where we will consider $\varepsilon >0$ small enough so that ${\bf s}_g + \varepsilon^{-2} {\bf s}_h $ is positive.  This is of course equivalent to finding solutions of the equation

\begin{equation}
-a_{m+n} \Delta_g u + ({\bf s}_g + \varepsilon^{-2} {\bf s}_h ) u =  \varepsilon^{-2} {\bf s}_h  u^{p_{m+n}-1}  
\end{equation}

Moreover, we can normalize $h$ and assume that ${\bf s}_h = a_{m+n}$. Then the previous equation is equivalent to:

\begin{equation}
-\varepsilon^2  \Delta_g u + ( ({\bf s}_g /a_{m+n} )  \varepsilon^2 + 1 ) u =  u^{p_{m+n}-1}  
\end{equation} 

Our goal is to find solutions of this equation for $\epsilon$ small enough.

Consider the functional $J_{\epsilon} : H^1 (M) \rightarrow \re$ given by

$$J_{\varepsilon} (u) = \varepsilon^{-n} \int_M \left( \frac{1}{2} \varepsilon^2 | \nabla u |^2  +  \frac{{\bf s}_g  \varepsilon^2 + a_{m+n} }{2a_{m+n}} u^2 -\frac{1}{p_{m+n}}  (u^+ )^{p_{m+n}} \right) dv_g  ,$$

\noindent
and the {\it Nehari manifold} $N_{\varepsilon}$ associated to the functional $J_{\varepsilon}$:

$$N_{\varepsilon} = \{ u\in H^1 (M) - \{ 0 \} : \int_M \varepsilon^2 | \nabla u|^2 + (({\bf s}_g /a_{m+n} )  \varepsilon^2 +  1)  u^2 dv_g = \int_M (u^+ )^{p_{m+n}} dv_g  \}  .$$

It is well known that critical points of $J_{\varepsilon}$ restricted to the Nehari manifold are positive solutions of Equation (6): it is clear that positive solutions of (6) belong to 
$N_{\varepsilon}$ and are critical points for $J_{\varepsilon}$. In the other direction if $u \in N_{\varepsilon}$ is a critical point of $J_{\varepsilon}$ then it follows that
$u$ must be nonnegative and a weak solution of (6); regularity theory then implies that $u$ is positive and smooth.

Let 

$${\bf m}_{\varepsilon} = \inf_{u\in N_{\varepsilon} } J_{\varepsilon} (u) = \varepsilon^{-n} (1/2 - 1/p_{m+n}) \inf_{u\in N_{\varepsilon} } \int (u^+  )^{p_{m+n}} dv_g .$$

Let us point out now the following:

\begin{Lemma} If $u\in N_{\varepsilon}$ then 

$$\int_M  \varepsilon^2 {\| \nabla u \|}^2 + (({\bf s}_g /a_{m+n} )  \epsilon^2 +  1)  u^2 dv_g  \geq \frac{2p}{p-2} \varepsilon^n {\bf m}_{\varepsilon} .$$

\end{Lemma}

The proof of Theorem 1.1 is based on the classical Lusternik-Schnirelmann theorems (see for instance \cite[Chapter 9]{Ambrosetti} or \cite{Cerami}) applied to
the functional $J_{\varepsilon} : N_{\varepsilon} \rightarrow \re$: 

\vspace{.4cm}

{\it Theorem: Let $J$ be a $C^1$  functional on a complete $C^{1,1}$ Banach manifold N. Assume that $J$ is
bounded below and satisfies the Palais-Smale condition. Let $J^d =\{ u \in N : J(u) <d \}$. Then $J$ has at least
$Cat(J^d )$ critical points in $J^d$.}

\vspace{.4cm}

$J_{\varepsilon}$ is clearly bounded below in $N_{\varepsilon}$ and it satisfies the Palais-Smale condition (since $p_{n+m}$ is subcritical).
We will call $\Sigma_{\varepsilon, d} = \{ u \in N_{\varepsilon} : J_{\varepsilon}  (u) < d \}$. 
We will prove

\begin{Theorem} There exists $\varepsilon_0 >0$ such that for any $\varepsilon \in (0, \varepsilon_0 )$ there exists $\delta >0$ such
that $Cat (\Sigma_{\varepsilon , {\bf m}_{\varepsilon} + \delta } ) \geq Cat (M) $. 
\end{Theorem}

Then we would have $Cat(M)$ critical points of $J_{\varepsilon}$ in $ \Sigma_{\varepsilon , {\bf m}_{\varepsilon} + \delta }$ and since $N_{\varepsilon}$ is
contractible we must have at least one other critical point (of higher energy) , giving the proof of Theorem 1.1. Then we have to prove Theorem 2.2. 

To prove Theorem 2.2 we will construct continuous maps ${\bf i}:M \rightarrow \Sigma_{\varepsilon , {\bf m}_{\varepsilon} + \delta }$ and
${\bf c}: \Sigma_{\varepsilon , {\bf m}_{\varepsilon} + \delta } \rightarrow M$ such that ${\bf c} \circ {\bf i}$ is homotopic
to the identity in $M$. If $A \subset \Sigma_{\varepsilon , {\bf m}_{\varepsilon} + \delta }$ is a contractible open subset  then ${\bf i}^{-1} (A) \subset M$ is
open and contractible (since it is homotopic to ${\bf c} (A)$). It follows that the existence of the maps ${\bf i}$ and ${\bf c}$ would imply 
that  $Cat (\Sigma_{\varepsilon , {\bf m}_{\varepsilon} + \delta } ) \geq Cat (M) $, proving  Theorem 2.2.

\subsection{Restricted Yamabe constants and ${\bf m}_{\varepsilon}$.} As in the previous subsection consider a closed Riemannian manifold $(N^m ,h)$ of constant positive scalar curvature and any closed
Riemannian manifold $(M^n ,g)$.  Let $\varepsilon >0$. One can restrict the $(g + \varepsilon^2 h)$-Yamabe functional  to functions which depend only on $M$ 
and  recall the definition in \cite{Akutagawa}

$$Y_M (N\times M, [\varepsilon^2 h +   g]) = \inf_{u\in H^1 (M) - \{ 0 \} } Y_{\varepsilon^2 h+  g} (u) $$

Let $g_E$ be the Euclidean metric on $\re^n$ and
recall also the following result from \cite[Theorem 1.1]{Akutagawa}

\begin{Theorem} $\lim_{\varepsilon \rightarrow 0} Y_M  (N\times M, [ \varepsilon^2 h + g ] = Y_{\re^n} (N \times \re^n , h+ g_E )$
\end{Theorem}

The constants  $ Y_{\re^n} (N \times \re^n , h+ g_E )$ can be computed \cite[Theorem 1.4]{Akutagawa}. To understand this,  recall the following particular case of the Gagliardo-Nirenberg
inequalities: for any $f\in H^1 (\re^n )$

$$ \| f \|_{p_{m+n}}^2 \leq  \sigma_{m,n} \| \nabla f \|_2^{\frac{2n}{m+n}} \| f \|_2^{\frac{2m}{m+n}} ,$$

\noindent
where $\sigma_{m,n} >0$ is the smallest constant for which the inequality holds for all $f$, the best constant. The function $U$ in Subsection 2.1 realizes
the inequality. This means that $U$ achieves the infimum of the functional

$$L(f) = \frac{ \|  \nabla f \|_2^{\frac{2n}{m+n}} \| f \|_2^{\frac{2m}{m+n}}}{\| f \|_{p_{m+n}}^2 } .$$

Note that for any $f \in H^1 (\re^n )$ and  $c, \lambda\in \re_{>0}$, we have that  $L(f) = L(cf_{\lambda})$, where $cf_{\lambda} (x) = c f(\lambda x)$. The Euler-Lagrange
equation of $L$ is:

\begin{equation}
- n \Delta u + m \frac{ \| \nabla u \|_2^2}{\| u \|_2^2 } u - (m+n) \frac{ \| \nabla u \|_2^2 }{\| u \|_{\bf p}^{\bf p}} u^{{\bf p}-1} =0 ,
\end{equation}

\noindent
where from now on we use the notation ${\bf p}=p_{m+n}$.

The function $U$ then verifies $n \|  U \|_2^2 = m \| \nabla U \|_2^2$ and   $n \|  U \|_{\bf p}^{\bf p} =( m +n)  \| \nabla U \|_2^2$

$$\frac{1}{\sigma_{m,n}} = L(U)=\frac{ n^{\frac{n}{m+n}}\  m^{\frac{m}{m+n}} \  \| U \|_{\bf p}^{\bf p}}{ (m+n) \| U \|_{\bf p}^2}$$

Then we get

$$\|  U \|_{\bf  p}^{{\bf p}-2}  = \frac{m+n}{\sigma_{m,n} n^{\frac{n}{m+n}} m^{\frac{m}{m+n}} } .$$

And then 

$${\bf m}(E) = \frac{{\bf p}-2}{2{\bf p}} \sqrt{\frac{(m+n)^{m+n}}{ (\sigma_{m,n})^{m+n} n^n m^m}} =   
\frac{1}{m+n} \sqrt{\frac{(m+n)^{m+n}}{ (\sigma_{m,n})^{m+n} n^n m^m}}.$$

From now on we will also use the notation ${\bf a}=a_{m+n}$.

If we let $V=Vol(N,h)$ then $V^{\frac{-2}{m}} h$ has volume one. And ${\bf s}_{V^{\frac{-2}{n}} h} = V^{\frac{2}{m}} {\bf  a}$. Then \cite[Theorem 1.4]{Akutagawa} says:

\begin{Theorem}

$$Y_{\re^n} (N \times \re^n , h+ g_E ) = \frac{ {\bf a} \  (m+n) \  V^{2/(m+n)}}{\sigma_{m,n} \  n^{n/(m+n)} \  m^{m/(m+n)}}  = {\bf a} V^{2/(m+n)} ((m+n) {\bf }({\bf m} (E)) )^{2/(m+n)} .$$

\end{Theorem}

\vspace{1cm}

Note that 

$$Y_{\varepsilon^2 h +g} (u) = Vol(N,\varepsilon^2 h)^{1-\frac{2}{\bf p}} \ \ \frac{ \int_M  {\bf a} {\| \nabla u \|}^2  +  ( \varepsilon^{-2} {\bf s}_h + {\bf s}_g ) u^2 dv_g }{ {\|  u  \|}_{\bf p}^2} $$

Then it folllows from Theorem 2.3 that

$$Y_{\re^n} (N \times \re^n , h+ g_E ) = \lim_{\varepsilon \rightarrow 0} Vol(N,\varepsilon^2 h)^{\frac{2}{m+n}} \inf_{ u\in H^1 ( M )  } \frac{\bf a}{\varepsilon^{2}} 
\frac{ \int_M  \varepsilon^2  {\| \nabla u \|}^2  +  ( 1+ \varepsilon^2 {\bf s}_g /{\bf a} ) u^2 dv_g }{ {\|  u  \|}_{\bf p}^2} .$$

Since the infimum on the right is realized by a positive function and the quotient is invariant under homothecies we can take the infimum
over $N_{\varepsilon}$ and we get:

$$=\lim_{\varepsilon \rightarrow 0} {\bf a}  \ \varepsilon^{\frac{-2n}{m+n} } V^{\frac{2}{m+n}} \inf_{ u\in N_{\varepsilon}} { \left( \int_M u^{\bf p} \right)}^{\frac{{\bf p}-2}{\bf p}} .$$

Therefore

$$Y_{\re^n} (N \times \re^n , h+ g_E ) = \lim_{\varepsilon \rightarrow 0}{\bf  a} \ V^{\frac{2}{m+n}}  \inf_{ u\in N_{\varepsilon}} \left( \varepsilon^{-n}   \int_M u^{\bf p} \right)^{\frac{2}{m+n}} $$

$$={\bf a}  \   V^{\frac{2}{m+n}}  (m+n)^{\frac{2}{m+n}}  \lim_{\varepsilon \rightarrow 0} {{\bf m}_{\epsilon} }^{\frac{2}{m+n}}. $$

Then it follows from Theorem 2.4 that:

\begin{Theorem}$\lim_{\varepsilon \rightarrow 0} {\bf m}_{\epsilon} = {\bf m}(E)$.
\end{Theorem}

\section{Concentration for functions in $\Sigma_{\varepsilon ,{\bf m}_{\varepsilon} +  \delta}$}

Recall that for  any $d >0$ we let $\Sigma_{\varepsilon, d} = \{
u\in N_{\varepsilon}  : J_{\varepsilon} (u) <  d \}$. 
In this section we will show that for $\varepsilon > 0$, $\delta >0$ small,  functions in $\Sigma_{\varepsilon , {\bf m}_{\varepsilon} + \delta}$
concentrate on a small ball.

For any function $u \in H^1 (M)$ such that $\int_M u^+ >0$ there exists a unique $\lambda (u) \in \re_{>0}$ such
that $\lambda (u) u \in N_{\varepsilon}$. Explicitly  one has

$$\lambda (u) = { \left(   \frac{\int_M \varepsilon^2 \| \nabla u \|^2 + ( \varepsilon^2 ({\bf s}_g  /{\bf a} ) +1 ) u^2  \  dv_g }{\int_M  (u^+ )^{\bf p} dv_g } \right) }^{\frac{1}{{\bf p}-2}} ,$$

\noindent
where we recall that we are using the notation ${\bf p}=p_{m+n}$, ${\bf a}=a_{m+n}$.

\begin{Proposition} Let $u \in N_{\varepsilon}$ such that it can be written as $u = u_1 + u_2$, where $u_1$ and $u_2$ have disjoint supports.
For $i=1,2$ let ${\| u_i^+  \|}_{\bf p}^{\bf p} =a_i \varepsilon^n  >0$  and $\int_M \varepsilon^2 \|  \nabla u_i  \|^2 + (\varepsilon^2 {\bf s}_g /{\bf a} + 1) u_i^2 \  dv_g =
 b_i \varepsilon^n $. 
Note that then $a_1 + a_2 = b_1 + b_2$ (since  $u \in N_{\varepsilon}$) .
Then

$$ J_{\varepsilon} (u) \geq {\bf m}_{\varepsilon} \left(  {\left( \frac{a_1}{b_1} \right) }^{\frac{2}{{\bf p}-2}} + {\left(  \frac{a_2 }{b_2} \right) }^{\frac{2}{{\bf p}-2}}  \right) .$$

\end{Proposition}

\begin{proof} Note that $\lambda (u_i )= (\frac{b_i }{a_i} )^{\frac{1}{p-2}}.$ Then from Lemma 2.1 we obtain

$$  {\left( \frac{b_i}{a_i} \right)}^{\frac{2}{{\bf p}-2}}  b_i \geq \frac{2{\bf p}}{{\bf p}-2}  \ \varepsilon^n \  {\bf m}_{\varepsilon} .$$

Then

$$ J_{\varepsilon} (u) = \frac{{\bf p}-2}{2{\bf p}} \varepsilon^{-n} (b_1 + b_2 )\geq \frac{{\bf p}-2}{2{\bf p}} \varepsilon^{-n}  \frac{2{\bf p}}{{\bf p}-2} \varepsilon^n {\bf m}_{\varepsilon} \left(
{\left( \frac{a_1}{b_1} \right)}^{\frac{2}{{\bf p}-2}}    +   {\left( \frac{a_2}{b_2} \right)}^{\frac{2}{{\bf p}-2}} \right). $$

\end{proof}

The proposition will be applied to show that if $J_{\varepsilon} (u)$ is close to ${\bf m}_{\varepsilon}$ and $u = u_1 + u_2$ as in the proposition then one of
the $u_i$'s must be very small.  

Let $u\in N_{\varepsilon}$. Assume that $J_{\varepsilon} (u) \leq 2 {\bf m}_{\varepsilon}$. Then if we have $u=u_1 + u_2$ as in the
proposition then $a_1 + a_2  = b_1 + b_2 \in [ \frac{2{\bf p}}{{\bf p}-2} {\bf m}_{\varepsilon} , 2   \frac{2{\bf p}}{{\bf p}-2} {\bf m}_{\varepsilon} ]$.

Now consider the following elementary observation:

\begin{Lemma} Let $r>0$ and $x_1 , x_2 , y_1 ,y_2$ be positive numbers such that $x_1 + x_2 = y_1 + y_2 \in [r,2r]$. Then 

$$ \varphi (x_1 , x_2 , y_1 , y_2 ) = {\left(  \frac{x_1}{y_1} \right) }^{\frac{2}{{\bf p} -2}} + {\left( \frac{x_2}{y_2}    \right)}^{\frac{2}{{\bf p} -2}} >1  .$$

For any $\delta \in (0,1)$ consider the set $A_{\delta} = \{ (x_1 , x_2 , y_1 , y_2 ) \in (\re_{>0} )^4  : x_1 + x_2 = y_1 + y_2 \in [r,2r], \  x_1 , x_2 \geq \delta r \}.$
And let 
$\Psi (\delta ) = \inf_{A_{\delta}} \varphi .$ Then $\Psi$ is a continuous function $\Psi : (0,1) \rightarrow (1 ,\infty )$, and $\Psi$ is
independent of $r$. 

\end{Lemma}

We can deduce

\begin{Corollary} Let $u\in N_{\varepsilon}$. Assume there are functions $u_1$, $u_2$ with disjoint supports such that $u= u_1 + u_2 $. Moreover assume that
$J_{\varepsilon} (u) \in [{\bf m}_{\varepsilon} , 2 {\bf m}_{\varepsilon}]$ and $\|  u_i^+ \|_{\bf p}^{\bf p} \geq \delta \varepsilon^n $ for some $\delta >0$, $i=1,2$. 
Then $J_{\varepsilon} (u) \geq \Psi (\delta )  \  {\bf m}_{\varepsilon} .$

\end{Corollary}

\vspace{1cm}
Given a function $u \in N_{\varepsilon}$ we will now construct, for any $x\in M$, a function $\overline{u} \in N_{\varepsilon}$ which is the sum of
two functions with disjoint supports as in the previous corollary. One of the functions will have support in a small ball around
point $x$.

\vspace{.5cm}

Let $r_0 (M,g)$ be a positive number such that for any $r< r_0 (M,g)$ and any $x \in M$ the geodesic ball $B(x,r) \subset (M,g)$  is
strongly convex. 

Fix $r < r_0 (M,g)$.

Fix a positive integer $l$ and another $i < l$. Consider the  bump functions $\varphi_{i,l}$ which is equal to
1 in $[0,\frac{(i-1) }{l} r ]$ and 0 in $[\frac{i}{l} r , \infty)$, and  $\sigma_{i,l}$ which is 1 in $[\frac{(i+1)}{l} r , \infty )$ and
0 in $[0 , \frac{i}{l} r ]$.
For any  $x\in M$, consider the   function $\varphi_{i,l}^{r,x} :  M \rightarrow [0,1]$ which
is 0 away from $B(x,r) \subset M$ and in $ B(x,r)$ (identified with the $r$-ball in $\re^n$) is given by
$\varphi_{i,l}$. Define  $\sigma_{i,l}^{r,x}$ in a similar way. Now for any function $u \in N_{\varepsilon}$ define
$u_{1,i,l}^{r,x} = \varphi_{i,l}^{r,x} \ u$ and $u_{2,i,l}^{r,x} = \sigma_{i,l}^{r,x} \ u$. 
Note that since $u^+ \neq 0$  we have that $ ( u_{1,i,l}^{r,x}  + u_{2,i,l}^{r,x} )^+ \neq 0$. 
Note also that  these 2 functions have disjoint supports.

Let

$$\overline{u}   = \overline{u}_{i,l}^{r,x} =     \lambda (  u_{1,i,l}^{r,x}  + u_{2,i,l}^{r,x}   ) \     (  u_{1,i,l}^{r,x}  + u_{2,i,l}^{r,x}   ) .$$

\begin{Lemma} For the closed Riemannian manifold $(M,g)$ fix $0< r <  r_0 (M,g) $ and  $\delta_0 >0$. There exists $\varepsilon_0 >0$ and $l \in {\bf N}$
 such that for any $\varepsilon \in (0, \varepsilon_0 )$ we have that 
 for any $u \in \Sigma_{\varepsilon , {\bf m}_{\varepsilon}  + \delta_0  }$ and any $x \in M$ there exist $i<l$ such that $ \overline{u}
\in \Sigma_{\varepsilon , {\bf m}_{\varepsilon} +  2  \delta_0 } $

\end{Lemma}

\begin{proof}
Assume from the beginning that $\varepsilon$ is small enough so that ${\bf m}_{\varepsilon}$ is close to ${\bf m} (E)$ (see Theorem 2.4).
Pick $\eta \in (0,1)$, close to 1,  such that 

$$({\bf m}_{\varepsilon}  + \delta_0  ) \times {\left( \frac{3 - 2 \eta }{\eta} \right) }^{\frac{2}{p-2}} (3 - 2 \eta ) < {\bf m}_{\varepsilon}  +  2 \delta_0.$$

Choose the positive integer $l$ large enough so that $6/l < 1-\eta .$

Consider  any $\varepsilon$ small and $u \in \Sigma_{\varepsilon , {\bf m}_{\varepsilon}  + \delta_0 }$. Pick any $x\in M$. 
Assume $l$ is even and divide $B(x,r)$ into $l/2$ annular regions of radius $(2/l)r$. Call $A_j$, $j=1,...,l/2$ each of these
regions. For any 2 nonnegative integrable functions on $B(x,r)$ their integrals are expressed as the sum of the integrals over
the $A_j$'s. Then for each of the two functions for at least $l/3$ of the $A_j$'s the integral must be at most 6/l of the total integral. Then there exists
$j$ such that the integral over $A_j$ of both functions is at most $6/l < 1-\eta$ of the total integral. We will
apply this to $(u^+ ) ^{\bf p}$ and $ \| \nabla u \|^2 $.  
We have that for some $i<l$ 

$$\|  u_{1,i,l}^{r,x}  + u_{2,i,l}^{r,x}    \|_{\bf p}^{\bf p} \geq \eta \| u \|_{\bf p}^{\bf p} $$

\noindent
and

$$\int_{A_i }  \|  \nabla u \|^2  dv_g  \leq  (1- \eta ) \int_M   \|  \nabla u \|^2 dv_g .$$

Note that  the gradients of $  u_{1,i,l}^{r,x} $ and $  u_{2,i,l}^{r,x} $ also have disjoint supports.
Also note that  $ u_{1,i,l}^{r,x}  + u_{2,i,l}^{r,x}  \leq  u$ and away from $A_i$, $u =  u_{1,i,l}^{r,x}  + u_{2,i,l}^{r,x} $.

We have

$$\int_{A_i} \|   \nabla   u_{1,i,l}^{r,x}  + u_{2,i,l}^{r,x}   \|^2  \  dv_g  =  \int_{A_i} \|  \nabla   u_{1,i,l}^{r,x}   \|^2   \  dv_g \ +  \     \int_{A_i}    \|  \nabla    u_{2,i,l}^{r,x} \|^2 
\ dv_g.$$

And

$$  \int_{A_i}  \|  \nabla   u_{1,i,l}^{r,x}   \|^2  \ dv_g   = \int_{A_i}  \|  \varphi \nabla u + u \nabla \varphi \|^2  \ dv_g \  \leq  \ 2 \int_{A_i} \| \varphi  \nabla u \|^2 +   \|  u \nabla \varphi \|^2   \ dv_g \ $$

$$ \leq 2 \int_{A_i}  \| \nabla u \|^2 + \frac{l^2}{r^2} u^2  \ dv_g .$$





A similar estimate can be carried out for $\int_{A_i} \| \nabla  u_{2,i,l}^{r,x}  \|^2   \ dv_g .$

Then we have that 

$$\int_M \varepsilon^2 \| \nabla  u_{1,i,l}^{r,x}  + u_{2,i,l}^{r,x}    \|^2 + (({\bf s}_g /{\bf a} )  \varepsilon^2 +  1) (  u_{1,i,l}^{r,x}  + u_{2,i,l}^{r,x} )^2  \ dv_g \  \leq $$

$$\int_{M-A_i} \varepsilon^2 \| \nabla u \|^2 +  (({\bf s}_g /{\bf a} )  \varepsilon^2 +  1) u^2 \  dv_g   \ \ \ \ + $$

$$\int_{A_i} 2  \varepsilon^2  \| \nabla u \|^2  +  (({\bf s}_g /{\bf a} )  \varepsilon^2 + 2 \varepsilon^2  (l/r)^2 +   1) u^2  \ dv_g $$

Therefore for the value of  $\eta$ we have chosen we can by pick first $l$ large enough and then $\varepsilon$ small enough to obtain 

$$\int_M \varepsilon^2 \| \nabla  ( u_{1,i,l}^{r,x}  + u_{2,i,l}^{r,x}  )   \|^2 + (({\bf s}_g /{\bf a} )  \varepsilon^2 +  1) ( u_{1,i,l}^{r,x}  + u_{2,i,l}^{r,x} )^2 \  dv_g \leq $$

$$ ( 3 -2\eta  ) \int_M \varepsilon^2 | \nabla u|^2 + (({\bf s}_g /{\bf a} )  \varepsilon^2 +  1) u^2  \ dv_g .$$

It follows that 

$$\lambda  (  u_{1,i,l}^{r,x}  + u_{2,i,l}^{r,x}  ) \leq   { \left( \frac{3- 2 \eta }{\eta}  \right)  }^{\frac{1}{{\bf p}-2}} . $$

Then, for $\overline{u} = \lambda (   u_{1,i,l}^{r,x}  + u_{2,i,l}^{r,x}  ) (   u_{1,i,l}^{r,x}  + u_{2,i,l}^{r,x}  ) \in N_{\varepsilon} $, 

$$ J_{\varepsilon} (\overline{u} ) = \frac{ \lambda^2  (u_1 + u_2 )} {\varepsilon^n}   \frac{p-2}{2p}     \int_M \varepsilon^2 \| \nabla (u_1 + u_2 ) \|^2 + (({\bf s}_g /a )  \varepsilon^2 +  1) 
(u_1 + u_2)^2  \ dv_g $$ 

$$ \leq  \varepsilon^{-n}   { \left( \frac{3- 2 \eta }{\eta}  \right)  }^{\frac{2}{p-2}}   \frac{p-2}{2p}   (3- 2\eta )
 \int_M \varepsilon^2 | \nabla u|^2 + (({\bf s}_g /a )  \varepsilon^2 +  1) u^2 dv_g   .$$

$$=   { \left( \frac{3- 2 \eta }{\eta}  \right)  }^{\frac{2}{p-2}}   (3- 2\eta ) J_{\varepsilon} (u) < {\bf m}_{\varepsilon} + 2 \delta_0 $$

\end{proof}

\begin{Remark}
On any closed Riemannian manifold for any $\varepsilon >0$  there is a set of points $x_j$, $j=1,...,K_{\varepsilon}$ such that
the balls $B(x_j , \varepsilon )$ are disjoint, and the set is maximal under this condition. It follows that the balls $B(x_j , 2 \varepsilon )$ cover $M$. 
It is easy to construct closed sets $A_j$, $ B(x_j ,  \varepsilon) \subset A_j \subset  B(x_j , 2 \varepsilon )$ which cover $M$ and only intersect in their 
boundaries. Moreover one can see by a volume comparison argument that if $\varepsilon$ is small enough there is a constant $K$, independent of 
$\varepsilon$,  such that for any point in $M$ can be in at most $K$ of the balls $B(x_j , 3\varepsilon )$.
\end{Remark}

Next we will see that  for $\varepsilon$ small for any $u \in N_{\varepsilon}$ there is a ball of radius $\varepsilon$ containing some
fixed part of the total integral if $ (u^+ )^{\bf p}$. The result is similar to \cite[Lemma 5.3]{Micheletti}, we include the proof for completeness.

\begin{Proposition} There exists $\gamma >0$ such that  there exists $\varepsilon_0 >0$ such that for any
$\varepsilon \in (0,\varepsilon_0 )$ if $u\in N_{\varepsilon}$ then
there exists a ball of radius $\varepsilon$, $B(x,\varepsilon ) \subset M$ such that

$$\varepsilon^{-n} \int_{B(x, \varepsilon )} (u^+ )^{\bf p} \geq \gamma $$

\end{Proposition}

\begin{proof}
Let us consider $\varepsilon_0 >0$ such that for $\varepsilon < \varepsilon_0$ one can construct sets like in the previous remark. 
Consider $u\in N_{\varepsilon}$. 
Let $u_j = u^+   \chi_{A_j}$ be the 
restriction of $u^+$ to $A_j$ (extended by 0 away from $A_j$ in order to consider it as a function on $M$). Then

$$\varepsilon^{-n} \int_M \varepsilon^2 \| \nabla u \|^2 + (1+ \varepsilon^2 {\bf s}_g/{\bf a}) u^2 dv_g = \varepsilon^{-n} \int_M (u^+ )^{\bf p} dv_g $$

$$=\Sigma_{j}  \  ( {\varepsilon^{- \frac{n({\bf p}-2)}{\bf p}}} {   \| u_j \|_{\bf p}^{{\bf p}-2} } ) \ (   { \varepsilon^{-\frac{2n}{\bf p}}}  { \| u_j \|_{\bf p}^2 } )$$

$$\leq \left( \max_j   {\varepsilon^{-\frac{n(p-2)}{\bf p}}}   {   \| u_j \|_{\bf p}^{{\bf p}-2} } \right) 
\Sigma_{j}   \  { \varepsilon^{-\frac{2n}{\bf p}}}   { \| u_j \|_{\bf p}^2 }   $$

Now let $\varphi_{\varepsilon}$ be the cut-off function on $\re^n$ which es 1 in $B(0, 2 \varepsilon )$ and vanishes away from
$B(0, 3 \varepsilon )$. $\| \nabla \varphi_{\varepsilon} \| = 1/\varepsilon$ in the intermediate annulus.

Define, for $j=1 \dots K_{\varepsilon}$,  

$$ u_{j, \varepsilon}  (x) = u^+  (x) \varphi_{\varepsilon} (d(x, x_j )).$$

Since $u_j \leq  u_{j, \varepsilon}  $ and by using the Sobolev inequalities we obtain that, for some constant $C$,

$$\varepsilon^{-\frac{2n}{\bf p}} \| u_j \|_{\bf p}^2 \leq   \varepsilon^{-\frac{2n}{\bf p}} \| u_{j , \varepsilon}   \|_{\bf p}^2  \leq C \varepsilon^{-n} \int_M \varepsilon^2 \| \nabla u_{j , \varepsilon} \|^2 
+ (1 + \varepsilon^2 {\bf s}_g /{\bf a})  (u_{j , \varepsilon}  )^2 dv_g  $$

Then, since we have that $u_{j, \varepsilon}  \leq u^+ $ and on $B(x_j , 3\varepsilon )- A_j $

$$\varepsilon^2 \| \nabla u_{j,\varepsilon}  \|^2  \leq 2\varepsilon^2 \| \nabla u^+ \|^2 + 2 (u^+ )^2 ,$$ 

\noindent
we obtain

$$\varepsilon^{-n} \int_M \varepsilon^2 \| \nabla u \|^2 + (1+ \varepsilon^2 {\bf s}_g/{\bf a}) u^2 dv_g \leq \Lambda  C \varepsilon^{-n} \ \ \times $$

$$
  \  \Sigma_{j}  \left(   \int_{A_j}  \varepsilon^2 \| \nabla u_j  \|^2 + \frac{{\bf a} +  \varepsilon^2 {\bf s}_g  }{\bf a } (u_j )^2 dv_g 
+
 \int_{C_j }  2 \varepsilon^2 \| \nabla u^+  \|^2 +  \frac{3{\bf a} +  \varepsilon^2 {\bf s}_g  }{\bf a } \  (u^+ )^2 dv_g  \right)  ,
 $$

\noindent
where $\Lambda =  \max_j   {\varepsilon^{-\frac{n({\bf p}-2)}{\bf p}}}   {   \| u_j^+ \|_{\bf p}^{{\bf p}-2} } $ and $C_j =  B(x_j , 3 \varepsilon ) -A_j$. Since any 
point $x\in M$ could be in at most $K$ of the $C_j$'s we have that the right hand side is

$$\leq  \Lambda  C \varepsilon^{-n} \ \   \times \ \     \int_{M}  \varepsilon^2 \| \nabla u^+  \|^2 + \frac{{\bf a} +  \varepsilon^2 {\bf s}_g  }{\bf a } (u^+  )^2 dv_g 
+ K  \int_{M }  2 \varepsilon^2 \| \nabla u^+  \|^2 +  \frac{3{\bf a} +  \varepsilon^2 {\bf s}_g  }{\bf a } \  (u^+ )^2 dv_g .$$

Then we have that 

$$ \int_M \varepsilon^2 \| \nabla u \|^2 + (1+ \varepsilon^2 {\bf s}_g/{\bf a}) u^2 dv_g \leq \Lambda  C (1+3K)  \int_M \varepsilon^2 \| \nabla u \|^2 
+ (1+ \varepsilon^2 {\bf s}_g/{\bf a}) u^2 .$$

Therefore  $ \Lambda  C (1+3K) \geq 1$ and the theorem follows with 

$$\gamma =  {  \left( \frac{1}{C(1+3K)} \right) }^{\frac{\bf p}{{\bf p}-2}}   .$$

\end{proof}

Finally we can deduce:

\begin{Theorem} Fix $r < r_0 (M,g)$. For any $\eta <1$ there exist $\varepsilon_0  , \delta_0 >0$ such that for any $\varepsilon \in (0 , \varepsilon_0 )$, $\delta
\in (0,\delta_0 )$ and $u\in \Sigma_{\varepsilon , {\bf m}_{\varepsilon} + \delta}$ there exists  $x \in M$ such that 
$\int_{B(x,r)}  (u^+ )^{\bf p} \geq \eta \int_M ( u ^+ )^{\bf p} $.

\end{Theorem}

\begin{proof} Assume the theorem is  not true.  Then there exist  exist $\eta <1$ and  sequences of positive numbers  $\varepsilon_j \rightarrow 0$, $\delta_j \rightarrow 0$, and
$u_j \in \Sigma_{\varepsilon_j , {\bf m}_{\varepsilon } + \delta_j }$   
such that for any
$x\in M$ one has $\int_{B(x,r)}   (u_j^+ )^{\bf p} < \eta \int_M (u_j^+ )^{\bf p} .$

For each $j$ large enough the previous proposition  provides  $x_j \in M$ such that, for some fixed $\gamma >0$, 

$$ \varepsilon_j^{-n} \int_{B(x_j , \varepsilon_j )}  (u_j^+ )^{\bf p} \geq \gamma .$$

Lemma 3.4  then gives a function $\overline{u_j} = u_{j,1} + u_{j,2}$  such that $\overline{u_j} \in  \Sigma_{\varepsilon_j , {\bf m}_{\varepsilon } + 2 \delta_j }$,
$ u_{j,1}$ is supported inside a ball centerd at $x_j$, $ u_{j,1}$ and $ u_{j,2}$ have disjoint support and $ u_j  =  \overline{u_j}$ in $B(x_j , \varepsilon_j )$ and
outside $B(x_j , r )$.

We have that 

$$ \varepsilon_j^{-n} \int_{M}  (u_{j,1}^+ )^{\bf p} \geq \gamma $$

\noindent
and

$$ \varepsilon_j^{-n} \int_{M}  (u_{j,2}^+ )^{\bf p}>  \varepsilon_j^{-n} (1 - \eta )   \int_M (u_j^+ )^{\bf p}    \geq (1- \eta ) \frac{2{\bf p}}{{\bf p} -2} {\bf m}_{\varepsilon}  .$$

Then  it follows from Corollary 3.3  that there exists $\delta_1 >0$, independent of $j$, such that
$J_{\varepsilon_j } (\overline{u_j}  ) \geq  \Psi (\delta_1 ) {\bf m}_{\varepsilon}$ .

But for $j$ large enough we have that $J_{\varepsilon_j } (\overline{u_j}  ) < {\bf m}_{\varepsilon} + 2\delta_j  < \Psi (\delta_1 ) {\bf m}_{\varepsilon}$,
reaching a
contradiction.

\end{proof}

\section{The inclusion ${\bf i} : M \rightarrow \Sigma_{\epsilon , {\bf m}_{\varepsilon} + \delta} $ }

For any $r>0$ consider the continuous piecewise linear map $\varphi_r : \re \rightarrow [0,1]$ such that
$\varphi_r (t) = 1$ if $t\leq r$, $\varphi_r (t) = 0$ if $t\geq 2r$, $\varphi_r ' =-1/r$ on $[r,2r]$. 
Let $\varphi_r^n :\re^n \rightarrow [0,1]$ be given by $\varphi_r^n (x) =\varphi_r (\| x \| )$. 

As in the previous sections  $(M,g)$ is a closed Riemannian manifold. Let $r_0 >0$ be such that for any $x\in M$ 
and $0< r < r_0$ the geodesic ball of radius $ r$ in $(M,g)$ is 
strongly convex. Let $U_{\epsilon ,r}= \varphi_r^n \ U_{\epsilon}$ (where $U_{\varepsilon}$ is the function 
defined in Subsection 2.1).  Take $r<r_0 /2$ and  for any $x\in M$
identify the geodesic ball $B(x,2r)$ with the ball $B(0,2 r) \subset \re^n$ by the exponential map; we obtain a function
$U^x_{\epsilon, r} : M \rightarrow \re_{\geq 0}$ (which is supported in $B(x,2r)$).
 Let $t^x_{\epsilon ,r} = \lambda (  U^x_{\epsilon, r} ) \in \re_{>0}$ so that
$t^x_{\epsilon ,r} \ U^x_{\epsilon, r} \in N_{\epsilon}$. Let  ${\bf i}_{r,\varepsilon} (x)$ = $t^x_{\epsilon ,r } \ U^x_{\epsilon, r}$

\begin{Theorem} Fix $\delta >0$ and $r< r_0 /2$. There exists $\varepsilon_0 >0$ such that for any
$\varepsilon < \varepsilon_0$ we have ${\bf i}_{r,\varepsilon} (x) \in \Sigma_{\varepsilon , {\bf m} (E) + \delta}$.

\end{Theorem}

\begin{proof} Fix $r< r_0 /2$. Since $M$ is compact there is a positive constant $H$ such that
on any geodesic ball of radius $r$ if we call $g_0$ the corresponding Euclidean metric then
$ (1/H)  g_0 \leq g \leq H g_0$. By taking $s<r$ small enough one can take the corresponding $H$ as
close to 1 as needed.  

The function $U_{\varepsilon}$ gets concentrated in 0 as $\varepsilon \rightarrow 0$. For $l=2$ or $l=p$, any fixed $s$, we have  

$$ \varepsilon^{-n} \int_{\re^n} U_{\varepsilon}^l = \int_{\re^n} U^l $$

$$ \varepsilon^{-n} \int_{\re^n -B(0,s)} U_{\varepsilon}^l = \int_{\re^n -B(0,s/\varepsilon )} U^l  
\rightarrow  0    \ \  ({\rm as} \   \varepsilon \rightarrow 0 )   $$

$$\varepsilon^{-n} \int_{\re^n}   \varepsilon^2 | \nabla U_{\varepsilon} |^2 = \int_{\re^n} | \nabla U |^2 $$

$$\varepsilon^{-n} \int_{\re^n -B(0,s)}   \varepsilon^2 | \nabla U_{\varepsilon} |^2 = \int_{\re^n -B(0,s/\varepsilon )} | \nabla U |^2 \rightarrow  0 \ \  ({\rm as} \   \varepsilon \rightarrow 0 )$$

Then $\lim_{\varepsilon \rightarrow 0} J_{\varepsilon} (U_{\varepsilon ,r} )= {\bf m}(E)$, and the integrals over 
$B(0,r)-B(0,s)$ converge to 0 uniformly. Then from the comments above picking $s$ small enough we
can make the integrals on the geodesic ball $B(x,s) \subset M$ as close as necessary to the corresponding integrals
on $\re^n$, and the theorem follows.

\end{proof}


\section{Center of mass}

Let $(M,g)$ be a closed Riemannian manifold. A subset $V \subset M$ is called strongly convex if any pair of points $x,y \in V$  are joined by a 
unique normal geodesic segment and the whole segment is contained in $V$. Since $M$ is closed the exists $r_0 >0$  so that for any
$x\in M$ and any $r \leq r_0$, the geodesic ball of radius $r$ centered at $x$, $B (x,r )$,  is
strongly convex.

Let $u \in L^1 (M)$ be nonnegative. Consider the function continuous $P_u : M \rightarrow \re$

$$P_u (x) = \int_M (d(x,y))^2  u(y) dv_g (y) .$$  

If $r$ is small enough 
and the support of $u$ is contained in $B(x,r)$  then H. Karcher and K. Grove
\cite{Grove} defined the Riemannian center of mass of the function $u$, which they call the center of mass of 
the measure given by $u \ dv_g$, as
the unique global minimum of the function $P_u$. Details can be found in  \cite[Section 1]{Karcher}: 
the function $P_u$ is strictly convex in a small ball and the minimum cannot be achieved outside such a ball. 
We will denote the Riemannian center of mass of such a function
$u$ by ${\bf cm} (u)$. If we denote by $L^{1,r} (M)$ the space of functions $u\in L^1 (M)$ with support in some
geodesic ball of radius $r$ (sufficiently small), note that ${\bf cm} : L^{1,r} (M) \rightarrow M$ is a continuous function:
$u \mapsto P_u$ defines a continuous map $L^1 (M) \rightarrow C^0 (M)$ with image in the family of functions with a unique
minimum. And the minimum depends continuously on the $C^0$-topology of such functions.

We are interested in extending the notion of Riemannian center of mass to functions which are not necessarily
supported in a small ball. This is not
possible in general since the function $P_u$ will in general have more than one minimum. But we will show that it is still possible
to give a good definition of center of mass for functions which are concentrated in small balls.

For a function $u \in L^1 (M)$ and a positive number $r$ let the {\it(u,r)-concentration function} be

$$ C_{u,r} (x)  =  \frac{\int_{B(x,r)} | u | dv_g}{\| u \|_1}$$

Note that $C_{u,r} : M \rightarrow [0,1]$, and it is a continuous function. Of course if $r\geq diam(M)$ then $C_{u,r} \equiv 1$, and
for any $x\in M$ we have $\lim_{r\rightarrow 0} C_{u,r} (x) =0$.

Now let
 {\it r-concentration coefficient} of $u$, $C_r (u)$, be the maximum of $C_{r,u}$: 

$$ C_r (u) = \sup_{x\in M} \frac{\int_{B(x,r)} | u | dv_g}{\| u \|_1}$$

For any  $\mu \in (0,1)$ let $L^1_{r, \mu}  (M,g) = \{ u\in L^1 (M) : C_r (u) > \mu \}$. We included $g$ in the notation since $C_{u,r}$, $C_r$ depend
on $g$. But we will consider $g$ fixed and write $L^1_{r, \mu}  (M,g) = L^1_{r, \mu}  (M)$ in the rest of the section.

For any $\eta \in (1/2 , 1)$ consider the piecewise linear continuous function  $\varphi_{\eta}: \re \rightarrow [0,1]$ such
that $\varphi_{\eta} (t) =0$ if $t\leq 1-\eta$, $\varphi_{\eta} (t) =1$ if $t \geq \eta$ and it is linear and increasing
in $[1-\eta , \eta ]$. 

Fix $r< (1/2)  r_0 $ and for any $u \in L^1_{r,\eta} (M)$  let 

$$\Phi_{r , \eta} (u)  (x) = \varphi_{\eta} (C_{u, r } (x) ) \ u(x) .$$

We have:

\begin{Lemma} For any $u\in L^{1}_{r,\eta} (M)$ the support of $ \Phi_{r , \eta} (u) $
is contained in a geodesic ball of radius $2r$ (centered at a point of maximal concentration).

\end{Lemma}

\begin{proof} Suppose that $x$ is any point where $C_{u,r}$ has a maximum. 
Then $C_{u,r} (x) \geq \eta$. If $d(x,y) >2r$ then $B(y,r) \cap B(x,r) = \emptyset$ and therefore  $C_{u,r} (y) \leq 1-\eta$. By the definition 
of $\varphi_{\eta}$ we have that $\varphi_{\eta} (C_{u, r } (y) )=0$ and the support of  $ \Phi_{r , \eta} (u) $ is
contained in $B(x,2r)$ .

\end{proof}

Then we have

\begin{Theorem} For any $r< (1/2) r_0$ and $\eta >1/2$ there exists a continuous function ${\bf Cm}(r,\eta ): L^{1}_{r,\eta} (M) \rightarrow M$, such that
if $x\in M$ verifies that  $C_{r,u} (x) >\eta$ then ${\bf Cm}( r,\eta ) (u) \in B(x, 2r)$
\end{Theorem}

\begin{Definition} Any function ${\bf Cm} (r,\eta ) (u)$ as in Theorem 3.2 will be called a {\it (r,$\eta $)-Riemannian center of mass} of $u$.
\end{Definition}

\begin{proof} We define ${\bf Cm} (r,\eta)) (u)= {\bf cm} (\Phi_{r , \eta} (u))$, which is well defined by
\cite{Grove, Karcher} and Lemma 5.1. It is clearly a continuous function. If
$x\in M$ verifies that  $C_{r,u} (x) >\eta$ then by Lemma 3.1 
the function $\Phi_{r , \eta} (u)$ is supprted in $B(x,2r)$. Then by \cite[Theorem 1.2]{Karcher}
${\bf cm} (\Phi_{r , \eta} (u)) \in B(x,2r)$.

\end{proof}

\section{ The map ${\bf c} $ and the proof  Theorem 2.2 }

Fix ${\bf r} < (1/2) \  r_0 (M,g)$.  
Consider for instance $\eta = 0.9$ and let $\varepsilon_0 , \delta_0$ be as in Theorem 3.7.  
Theorem 4.1 and Lemma 2.1  imply that there exists $\varepsilon_1  \in (0,  \varepsilon_0 )$ such that for 
any $\varepsilon \in (0, \varepsilon_1 )$  we have a continuous map ${\bf i}_{\varepsilon }
: M \rightarrow \Sigma_{\varepsilon , {\bf m}_{\varepsilon} + \delta_0 }$. 

Since $\varepsilon < \varepsilon_0$ Theorem 3.7 says that if $u \in \Sigma_{\varepsilon , {\bf m}_{\varepsilon} + \delta_0  }$ then
$(u^+ )^p \in  L^1_{{\bf r} , 0.9 }$. Then by Theorem 5.2 we have a continuous map 
${\bf c}_{\varepsilon} : \Sigma_{\varepsilon , {\bf m}_{\varepsilon} + \delta_0 } \rightarrow M$,
${\bf c}_{\varepsilon} = {\bf Cm} ({\bf r}, 0.9 ) ( (u^+ )^p )$ .

For any $x\in M$ we have by construction that the support of ${\bf i}_{\varepsilon} (x)$ is contained
in $B(x, 2{\bf r}) $. Since $2 {\bf r} < r_0 (M,g)$ we have that  ${\bf c}_{\varepsilon}  \circ {\bf i}_{\varepsilon} (x) \in B(x, 2{\bf r} ) .$ 
Then we obtain a homotopy between  ${\bf c}_{\varepsilon}  \circ {\bf i}_{\varepsilon} : M \rightarrow M$ and $Id_M$ by following
the unique minimizing geodesic from ${\bf c}_{\varepsilon}  \circ {\bf i}_{\varepsilon} (x)$ to $x$. As mentioned in Subsection 2.2, 
this implies that $Cat (  \Sigma_{\varepsilon , {\bf m}_{\varepsilon} + \delta_0 } ) \geq Cat (M)$, proving Theorem 2.2 and therefore
Theorem 1.1.

\end{document}